\documentclass[12pt]{amsart}
\usepackage{macros}
\standardsettings

\newcommand{\name}[2]{#2}

\newcommand{\irredabi}{$\{j\in D: a_j \neq b_j \} \subset \{j\in D : j \succ i\}$}

\draftfalse

\begin{document}
\title[Badly approximable points on self-affine sponges]{Badly approximable points on self-affine sponges and the lower Assouad dimension}

\begin{Abstract}
We highlight a connection between Diophantine approximation and the lower Assouad dimension by using information about the latter to show that the Hausdorff dimension of the set of badly approximable points that lie in certain non-conformal fractals, known as self-affine sponges, is bounded below by the dynamical dimension of these fractals. In particular, for self-affine sponges with equal Hausdorff and dynamical dimensions, the set of badly approximable points has full Hausdorff dimension in the sponge. Our results, which are the first to advance beyond the conformal setting, encompass both the case of Sierpi\'nski sponges/carpets (also known as Bedford--McMullen sponges/carpets) and the case of Bara\'nski carpets. We use the fact that the lower Assouad dimension of a hyperplane diffuse set constitutes a lower bound for the Hausdorff dimension of the set of badly approximable points in that set.
\end{Abstract}
\authortushar\authorlior\authordavid\authormariusz


\maketitle

Fix $d\in\N$. Dirichlet's theorem in Diophantine approximation states that for all $\xx\in\R^d$, there exist infinitely many rational points $\pp/q\in\Q^d$ such that
\[
\left\|\xx - \frac{\pp}{q}\right\| < \frac{1}{q^{1 + 1/d}}\cdot
\]
A point $\xx\in\R^d$ is said to be \emph{badly approximable} if this inequality cannot be improved by more than a constant, i.e. if there exists a constant $c > 0$ such that for any rational point $\pp/q\in\Q^d$, we have
\[
\left\|\xx - \frac{\pp}{q}\right\|\geq \frac{c}{q^{1 + 1/d}}\cdot
\]
We denote the set of all badly approximable points in $\R^d$ by $\BA_d$. Dirichlet's theorem shows that in some sense, out of all the points in $\R^d$, badly approximable points are the hardest to approximate by rationals. It is well-known that $\BA_d$ is a Lebesgue nullset of full Hausdorff dimension in $\R^d$, see e.g. \cite[Chapter III]{Schmidt3}.

For more than a decade now, as part of the burgeoning study of Diophantine properties of fractal sets and measures \cite{KLW,EinsiedlerTseng,EFS,FishmanSimmons1,DFSU_GE1}, there has been a growing interest in computing the Hausdorff dimension of the intersection of $\BA_d$ with various fractal sets. Since $\BA_d$ has full dimension, one expects its intersection with any fractal set $J \subset \R^d$ to have the same dimension as $J$, and this can be proven for certain broad classes of fractal sets $J$, see e.g. \cite{BFKRW,DFSU_BA_conformal,Fishman,KleinbockWeiss1}.

However, progress so far has been limited to the class of fractals defined by conformal dynamical systems, and it has been a natural challenge to understand what happens beyond this case. Non-conformal dynamical systems (where the system is expanding but may have different rates of expansion in different directions) are often much more complicated than conformal ones, which can often be thought of as essentially the same as one-dimensional systems. For instance, the Hausdorff dimension of any conformal expanding repeller can be computed via Bowen's formula (e.g. \cite[Corollary 9.1.7]{PrzytyckiUrbanski}), but it is far more difficult to compute the Hausdorff dimension of even relatively simple non-conformal fractals, such as the limit sets/measures of affine iteration function systems (IFSes) satisfying the open set condition. To make progress one generally has to assume either some randomness in the contractions defining the IFS, as in \cite{Falconer2,Kaenmaki2}, or some special relations between these contractions, as in \cite{Baranski,DasSimmons1}. An exception to this is a recent theorem of \name{Bal\'asz}{B\'ar\'any} and \name{Antti}{K\"aenm\"aki} \cite{BaranyKaenmaki}, who showed that every self-affine measure on the plane is exact dimensional.

In this paper we will concentrate on the latter situation, considering the class of \emph{self-affine sponges}, and in particular analyzing the Hausdorff dimension of the intersection of a self-affine sponge with the set of badly approximable points. The class of self-affine sponges is the generalization to higher dimensions of the class of \emph{self-affine carpets}, which consists of subsets of $\R^2$ defined according to a certain recursive construction where each rectangle in the construction is replaced by the union of several rectangles contained in that rectangle (cf. Definition \ref{definitionsponges} below). The Hausdorff dimension of certain self-affine carpets was computed independently by \name{Tim}{Bedford} \cite{Bedford} and \name{Curt}{McMullen} \cite{McMullen_carpets}, and their results were extended by several authors \cite{LalleyGatzouras, KenyonPeres, Baranski, DasSimmons1}.

Given a self-affine sponge, we would like to know what the Hausdorff dimension of its intersection with $\BA_d$ is. There are two subtleties that make this question more difficult to answer than in the conformal case. One involves the question of what hypotheses are sufficient to deduce that a self-affine sponge intersects $\BA_d$ nontrivially. In the case of self-conformal sets, the answer has always turned out to be an irreducibility assumption: in the most general case, that the set in question is not contained in any real-analytic manifold of dimension strictly less than $d$ (see \cite{DFSU_BA_conformal}). This assumption is natural because of a well-known obstruction: any point contained in a rational affine hyperplane cannot be badly approximable, and thus any set that intersects $\BA_d$ cannot be contained in a rational affine hyperplane. Strengthening this requirement from rational hyperplanes to all hyperplanes, and then from hyperplanes to real-analytic manifolds, is natural from a geometric point of view. However, in the case of self-affine sponges the irreducibility assumption needs to be stronger than the condition of not being contained in a manifold; in the next section we will say precisely what assumption is needed (see Definition \ref{definitionirreducible}). Our irreducibility assumption is satisfied in ``most'' examples, and we can show that standard techniques (i.e. Schmidt's game and hyperplane diffuseness) must fail for sponges that are not irreducible in our sense (see Proposition \ref{propositionnoHDsubsets}).

The other subtlety is that self-affine sponges may have no natural measure of full dimension, as recently discovered by two of the authors \cite{DasSimmons1}. For such sponges, our techniques cannot prove that the intersection of the sponge with $\BA_d$ has full dimension in the sponge, but only that the dimension of this intersection is bounded below by the \emph{dynamical dimension} of the sponge, i.e. the supremum of the dimensions of the invariant measures. Regarding this difficulty, we leave open the possibility that it may still be possible to prove the full Hausdorff dimension of the sponge's intersection with $\BA_d$ using Schmidt's game and hyperplane diffuseness, but new ideas would be needed. This problem as well as a few other open problems are listed at the end of the paper.\\

{\bf Acknowledgements.} The first-named author was supported in part by a 2016-2017 Faculty Research Grant from the University of Wisconsin--La Crosse. The second-named author was supported in part by the Simons Foundation grant \#245708. The third-named author was supported in part by the EPSRC Programme Grant EP/J018260/1. The fourth-named author was supported in part by the NSF grant DMS-1361677. The authors thank the anonymous referee for helpful comments.

\section{Main results}

We use the same notation to describe self-affine sponges as in \cite{DasSimmons1}:

\begin{definition}[{\cite[Definitions 2.1 and 2.2]{DasSimmons1}}]
\label{definitionsponges}
Fix $d \geq 1$, and let $D = \{1,\ldots,d\}$. For each $i \in D$, let $A_i$ be a finite index set, and let $\Phi_i = (\phi_{i,a})_{a\in A_i}$ be a finite collection of contracting similarities of $[0,1]$, called the \emph{base IFS in coordinate $i$}. (Here IFS is short for \emph{iterated function system}.) Let $A = \prod_{i\in D} A_i$, and for each $\aa = (a_1,\ldots,a_d) \in A$, consider the contracting affine map $\phi_\aa : [0,1]^d \to [0,1]^d$ defined by the formula
\begin{equation*}
\phi_\aa(x_1,\ldots,x_d) = (\phi_{\aa,1}(x_1),\ldots,\phi_{\aa,d}(x_d)),
\end{equation*}
where $\phi_{\aa,i}$ is shorthand for $\phi_{i,a_i}$ in the formula above, as well as elsewhere. Geometrically, $\phi_\aa$ can be thought of as corresponding to the rectangle to which it sends $[0,1]^d$:
\[
\phi_\aa([0,1]^d) = \prod_{i\in D} \phi_{\aa,i}([0,1]) \subset [0,1]^d .
\]
Given $E \subset A$, we call the collection $\Phi \df (\phi_\aa)_{\aa\in E}$ a \emph{diagonal IFS}. It is a special case of the more general notion of an \emph{affine IFS}. The \emph{coding map} of $\Phi$ is the map $\pi:E^\N \to [0,1]^d$ defined by the formula
\[
\pi(\omega) = \lim_{n\to\infty} \phi_{\omega\given n}(\0),
\]
where $\phi_{\omega\given n} \df \phi_{\omega_1}\circ\cdots\circ\phi_{\omega_n}$. Finally, the \emph{limit set} of $\Phi$ is the set $\Lambda_\Phi \df \pi(E^\N)$. We call the limit set of a diagonal IFS a \emph{self-affine sponge}. If $d = 2$, the limit set is also called a \emph{self-affine carpet}.

The sponge $\Lambda_\Phi$ is called \emph{Bara\'nski} (resp. \emph{strongly Bara\'nski}) if the base IFSes all satisfy the open set condition (resp. the strong separation condition) with respect to the interval $\I = (0,1)$ (resp. $\I = [0,1]$), i.e. if for all $i\in D$, the collection
\[
\big(\phi_{i,a}(\I)\big)_{a\in A_i}
\]
is disjoint.
\end{definition}

We now define the notion of irreducibility that we need in order to state our theorem:

\begin{definition}
\label{definitionirreducible}
Define a partial order $\prec$ on $D$ by writing $i \prec j$ if $|\phi_{\aa,i}'| \geq |\phi_{\aa,j}'|$ for all $\aa\in E$. In other words, $i \prec j$ if all contractions of $\Phi$ contract at least as fast in coordinate $j$ as in coordinate $i$. The sponge $\Lambda_\Phi$ is said to be \emph{irreducible} if for all $i\in D$, there exist $\aa,\bb\in E$ such that $a_i \neq b_i$ but \irredabi. In other words, $\Phi$ is irreducible if for every coordinate, $\Phi$ contains two contractions that can be distinguished in that coordinate but not in any coordinate that allows for slower contraction than in the original coordinate.

The sponge $\Lambda_\Phi$ is said to have \emph{distinguishable coordinates} if for all $i,j\in D$ with $i\neq j$, there exists $\aa\in E$ such that $|\phi_{\aa,i}'| \neq |\phi_{\aa,j}'|$. Note that in two dimensions, a carpet has distinguishable coordinates if and only if its IFS does not consist entirely of similarities.
\end{definition}

%

\begin{example}
\label{examplecarpet}
Let $d = 2$, $m_1 = 3$, and $m_2 = 2$, and consider the base IFSes
\begin{align*}
\Phi_i &= (\phi_{i,a})_{0 \leq a \leq m_i - 1},&
\phi_{i,a}(x) &= \frac{a + x}{m_i}\cdot
\end{align*}
The product IFS consists of the affine contractions
\[
\phi_\aa(x_1,x_2) = \left(\frac{a_1 + x_1}{3},\frac{a_2 + x_2}{2}\right), \;\; \aa\in A = \{0,1,2\}\times\{0,1\}.
\]
Let $E^{(1)} = \{(0,0),(1,1),(2,0)\} \subset A$ and $E^{(2)} = \{(0,0),(2,1)\} \subset A$, and consider the diagonal IFSes $\Phi^{(k)} = (\phi_\aa)_{\aa\in E^{(k)}}$ ($k = 1,2$). (Cf. Figure \ref{figurecarpet}.) The IFS $\Phi^{(1)}$ is irreducible, since for $i = 1$ we can take $\aa = (0,0)$ and $\bb = (2,0)$ and for $i = 2$ we can take $\aa = (0,0)$ and $\bb = (1,1)$. On the other hand, the IFS $\Phi^{(2)}$ is reducible, since if $i = 1$, then there do not exist $\aa,\bb\in E^{(2)}$ such that $a_i \neq b_i$ and \irredabi. Both of these IFSes have distinguishable coordinates.

We note that although $\Phi^{(2)}$ is reducible, its limit set is not contained in any line or smooth curve in $\R^2$. This contrasts with the case of limit sets of conformal IFSes, where a set is defined to be irreducible if it is not contained in any real-analytic manifold of dimension strictly smaller than the ambient dimension. The reason that we call $\Phi^{(2)}$ reducible is that its limit set does not have any hyperplane diffuse subsets, meaning that Schmidt's game cannot be used to deduce lower bounds on the dimension of its intersection with $\BA_d$; see Section \ref{sectiongames} and Proposition \ref{propositionnoHDsubsets} for details.
\end{example}

\begin{figure}
\scalebox{0.5}{
\begin{tikzpicture}[line cap=round,line join=rounr]
\clip(-1,-1) rectangle (7,7);
\fill[gray] (0,0) -- (0,3) -- (2,3) -- (2,0) -- cycle;
\fill[gray] (2,3) -- (2,6) -- (4,6) -- (4,3) -- cycle;
\fill[gray] (4,0) -- (4,3) -- (6,3) -- (6,0) -- cycle;
\draw (0,0)-- (0,6);
\draw (2,0)-- (2,6);
\draw (4,0)-- (4,6);
\draw (6,0)-- (6,6);
\draw (0,0)-- (6,0);
\draw (0,3)-- (6,3);
\draw (0,6)-- (6,6);
\end{tikzpicture}}
\scalebox{0.5}{
\begin{tikzpicture}[line cap=round,line join=rounr]
\clip(-1,-1) rectangle (7,7);
\fill[gray] (0,0) -- (0,3) -- (2,3) -- (2,0) -- cycle;
\fill[gray] (4,3) -- (4,6) -- (6,6) -- (6,3) -- cycle;
\draw (0,0)-- (0,6);
\draw (2,0)-- (2,6);
\draw (4,0)-- (4,6);
\draw (6,0)-- (6,6);
\draw (0,0)-- (6,0);
\draw (0,3)-- (6,3);
\draw (0,6)-- (6,6);
\end{tikzpicture}}
\caption{Generating templates for the IFSes $\Phi^{(1)}$ and $\Phi^{(2)}$ considered in Example \ref{examplecarpet}. The left-hand picture is irreducible but the right-hand picture is reducible. Both IFSes have distinguishable coordinates and are Bara\'nski but not strongly Bara\'nski.}
\label{figurecarpet}
\end{figure}
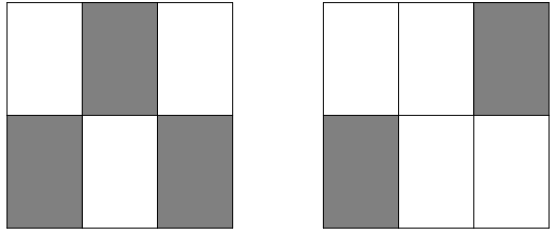

The last notion we need to define before we can state our theorem is the notion of the dynamical dimension of a self-affine sponge:

\begin{definition}[{\cite[Definition 2.6]{DasSimmons1}}]
\label{definitionDD}
The \emph{dynamical dimension} of a self-affine sponge $\Lambda_\Phi$ is the number
\[
\DynD(\Phi) \df \sup_\mu\{\HD(\pi_*[\mu])\},
\]
where the supremum is taken over all probability measures $\mu$ on $E^\N$ that are invariant under the shift map.\Footnote{By \cite[Theorem 2.7 and (2.13)]{DasSimmons1}, the dynamical dimension is the same if we take the supremum only over ergodic measures of positive entropy. This means that our definition agrees with the usual one in the literature e.g. \cite[(2.14)]{DenkerUrbanski2}.} Here, $\pi_*[\mu]$ denotes the pushforward of $\mu$ under the coding map $\pi$, and the \emph{Hausdorff dimension} of $\pi_*[\mu]$ is the infimum of the Hausdorff dimensions of sets that have full measure under $\pi_*[\mu]$.
\end{definition}

The dynamical dimension of a self-affine sponge is always bounded above by the Hausdorff dimension of the sponge. Equality holds in the case of Sierpi\'nski sponges (in which the coordinatewise rates of contraction are the same for all contractions in the IFS) \cite{KenyonPeres} and also in the case of Bara\'nski carpets (i.e. two-dimensional Bara\'nski sponges) \cite{Baranski}. In general, the dynamical dimension may be strictly less than the Hausdorff dimension, see \cite{DasSimmons1}; this is true even for three-dimensional Bara\'nski sponges satisfying the coordinate ordering condition (see Definition \ref{definitionLG} below).

We are now ready to state our main result:

\begin{theorem}
\label{maintheorem}
Let $\Lambda_\Phi \subset [0,1]^d$ be an irreducible Bara\'nski sponge with distinguishable coordinates. Then
\[
\HD(\Lambda_\Phi \cap \BA_d) \geq \DynD(\Lambda_\Phi).
\]
In particular, if $d = 2$ or if $\Lambda_\Phi$ is a Sierpi\'nski sponge, then $\Lambda_\Phi \cap \BA_d$ has full Hausdorff dimension in $\Lambda_\Phi$.
\end{theorem}

The idea of the proof can be described succinctly as follows. Let $\nu = \pi_*[\mu]$ be the image under the coding map of an ergodic shift-invariant probability measure $\mu$ on $E^\N$. Let $N$ be a large number, and let $F \subset E^N$ be a subset consisting of ``$\mu$-typical'' words. Then the limit set of the IFS $\Psi_F$ corresponding to $F$ can be shown to intersect $\BA_d$ in a set of dimension close to the dimension of $\nu$,  Roughly, this is because the elements of $\Psi_F$ are all relatively ``homogeneous'' and so the lower Assouad dimension of the limit set of $\Psi_F$, which is a lower bound for the dimension of its intersection with $\BA_d$ (see Proposition \ref{propositionfishman}), is close to the Hausdorff dimension of the limit set of $\Psi_F$, which is in turn close to the Hausdorff dimension of $\nu$ (computed using a Ledrappier--Young-type formula \cite[(2.13)]{DasSimmons1}). Finally, $\nu$ can be chosen so that its Hausdorff dimension is close to the dynamical dimension of $\Lambda_\Phi$.

Actually, we do not need to deal with all possible ergodic shift-invariant measures; it suffices to consider the smaller class of Bernoulli measures. A \emph{Bernoulli measure} is a measure of the form $\nu_\pp = \pi_*[\pp^\N]$, where $\pp$ is a probability measure on $E$. In \cite[Theorem 2.7]{DasSimmons1}, it was shown that the supremum of the Hausdorff dimensions of the Bernoulli measures is equal to the dynamical dimension, so in the above proof sketch $\nu$ can be assumed to be a Bernoulli measure. The function sending a Bernoulli measure to its Hausdorff dimension is continuous \cite[Theorem 2.9]{DasSimmons1}, so by compactness there exists a (not necessarily unique) Bernoulli measure whose Hausdorff dimension is equal to the dynamical dimension.

If $\nu_\pp$ is a Bernoulli measure of maximal dimension, then to show that the conclusion of Theorem \ref{maintheorem} holds it suffices to check that the above proof sketch can be made rigorous for the measure $\nu = \nu_\pp$. It turns out that the conditions under which this is possible are more general than the hypotheses of Theorem \ref{maintheorem}. To make this statement precise, we introduce a ``local'' analogue of the partial order $\prec$ considered in Definition \ref{definitionirreducible}:

\begin{definition}
\label{definitionmuirreducible}
Let $\pp$ be a probability measure on $E$, and for each $i \in D$ we define the \emph{Lyapunov exponent of $\pp$ in coordinate $i$} to be the number
\[
\chi_i(\pp) = -\int \log|\phi_{\aa,i}'| \;\dee\pp(\aa).
\]
Define a partial order $\prec_\pp$ on $D$ by writing $i \prec_\pp j$ if $\chi_i(\pp) \leq \chi_j(\pp)$.  The measure $\pp$ is said to be \emph{irreducible} (or equivalently, the sponge $\Lambda_\Phi$ is said to be \emph{irreducible with respect to $\pp$}) if for all $i\in D$, there exist $\aa,\bb\in E$ such that $a_i \neq b_i$ but $\{j\in D: a_j \neq b_j \} \subset \{j\in D : j \succ_\pp i\}$. Note that $\prec_\pp$ is a finer partial order than $\prec$, so if a sponge is irreducible in the sense of Definition \ref{definitionirreducible} then it is irreducible with respect to every probability measure on $E$.

The measure $\pp$ is said to have \emph{distinct Lyapunov exponents} if the numbers $\chi_i(\pp)$ ($i\in D$) are all distinct.
\end{definition}


We also introduce a ``local'' analogue of the Bara\'nski condition:

\begin{definition}[Cf. {\cite[Definition 3.1]{DasSimmons1}}]
\label{definitiongood}
Let $\Lambda_\Phi$ be a self-affine sponge, and let $I \subset D$ be a coordinate set. Let
\[
\Phi_I = (\phi_{I,\aa})_{\aa\in \pi_I(E)},
\]
where $\phi_{I,\aa}:[0,1]^I\to [0,1]^I$ is defined by the formula
\[
\phi_{I,\aa}(\xx) = \big(\phi_{\aa,i}(x_i)\big)_{i\in I}
\]
and $\pi_I:A\to A_I \df \prod_{i\in I} A_i$ is the projection map. We call $I$ \emph{good} (resp. \emph{strongly good}) if the collection
\[
\big(\phi_{I,\aa}(\I^I)\big)_{\aa\in \pi_I(E)}
\]
is disjoint, where $\I = (0,1)$ (resp. $\I = [0,1]$). Also, a measure $\pp$ on $E$ is called \emph{good} (resp. \emph{strongly good}) if for every $x > 0$, the set
\[
I(\pp,x) = \{i\in D : \chi_i(\pp) \leq x\}
\]
is good (resp. strongly good).
\end{definition}

We can now state a ``local'' version of Theorem \ref{maintheorem}:

\begin{theorem}
\label{maintheorem2}
Let $\Lambda_\Phi \subset [0,1]^d$ be a self-affine sponge, and let $\pp$ be an irreducible good probability measure on $E$ with distinct Lyapunov exponents. Then
\[
\HD(\Lambda_\Phi\cap \BA_d) \geq \HD(\nu_\pp).
\]
In particular, if $\HD(\nu_\pp) = \HD(\Lambda_\Phi)$, then $\Lambda_\Phi\cap \BA_d$ has full Hausdorff dimension in $\Lambda_\Phi$.
\end{theorem}

If $\Lambda_\Phi$ is an irreducible Bara\'nski sponge with distinguishable coordinates, then every probability measure on $E$ is both irreducible and good, and the set of measures with distinct Lyapunov exponents forms an open dense set. Thus, Theorem \ref{maintheorem2} implies Theorem \ref{maintheorem}.\\

{\bf Outline of the paper.} In the next section we recall some known results about the dimension of intersection of $\BA_d$ with fractals and its relation to the lower Assouad dimension, and state a strengthening of Theorem \ref{maintheorem2}, namely Theorem \ref{maintheorem3}. In Section \ref{sectionLG}, we prove some results which suffice to give a useful estimate on the Hausdorff dimension of $\Lambda_{\Psi_F} \cap \BA_d$, where $\Psi_F$ is a ``homogeneous'' IFS as described in our proof sketch above. In Section \ref{sectionproof} we use this estimate to prove Theorem \ref{maintheorem3}.

\begin{convention}
The symbols $\lesssim$, $\gtrsim$, and $\asymp$ will denote coarse multiplicative asymptotics. For example, $A\lesssim B$ means that there exists a constant $C > 0$ (the \emph{implied constant}) such that $A\leq C B$.
\end{convention}

\section{Schmidt's game, hyperplane diffuse sets, and the lower Assouad dimension}
\label{sectiongames}

For $d \geq 2$, the full dimension of $\BA_d$ in $\R^d$ was proven by \name{Wolfgang}{Schmidt} \cite{Schmidt2} using a technique now known as Schmidt's game. Since then, Schmidt's game and its variants have been used to prove the full dimension intersection of $\BA_d$ with various fractals, as well as various stability properties such as $\CC^1$ incompressibility \cite{BFKRW}. The modern approach \cite{BFKRW} is to first show that $\BA_d$ is winning for a variant of Schmidt's game known as the hyperplane absolute game, and then show that any set winning for the hyperplane absolute game is also winning for Schmidt's game played on any fractal satisfying a certain geometric condition called hyperplane diffuseness. The Hausdorff dimension of a set winning for Schmidt's game played on a fractal can be bounded from below based on the geometry of that fractal; there are currently two known ways of doing this, one based on the dimensions of fully supported doubling Frostman measures on the fractal \cite[Theorem 1.1]{KleinbockWeiss1}, and the other based on the lower Assouad dimension of the fractal \cite[Theorem 3.1]{Fishman}. The two methods give the same bound whenever the fractal in question is Ahlfors regular,\Footnote{We recall that a measure $\mu$ on $\R^d$ is said to be \emph{Ahlfors $s$-regular} if for all $\xx\in \Supp(\mu)$ and $0 < \rho \leq 1$, we have \[\mu(B(\xx,\rho)) \asymp \rho^\delta.\] A closed set $K \subset \R^d$ is said to be \emph{Ahlfors $s$-regular} if there exists an Ahlfors $s$-regular measure $\mu$ such that $K = \Supp(\mu)$, and \emph{Ahlfors regular} if it is Ahlfors $s$-regular for some $s > 0$.} which is true in most applications that have been considered so far. However, in our case the fractals are not Ahlfors regular, and there is a difference between the two methods. We work with the second method, based on the lower Assouad dimension, as it seems to be more suited to the situation we consider here.

\begin{definition}
Let $K$ be a closed subset of $\R^d$. For any $0 < \alpha,\beta < 1$, \emph{Schmidt's $(\alpha,\beta)$-game} is an infinite game played by two players, Alice and Bob, who take turns choosing balls in $\R^d$ whose centers lie in $K$, with Bob moving first. The players must choose their moves so as to satisfy the relations
\[
B_1 \supset A_1 \supset B_2 \supset \cdots
\]
and
\[
\rho(A_k) = \alpha \rho(B_k) \text{ and } \rho(B_{k + 1}) = \beta \rho(A_k) \text{ for } k\in\N,
\]
where $B_k$ and $A_k$ denote Bob's and Alice's $k$th moves, respectively, and where $\rho(B)$ denotes the radius of a ball $B$. Since the sets $B_1,B_2,\ldots$ form a nested sequence of nonempty closed sets whose diameters tend to zero, it follows that the intersection $\bigcap_k B_k$ is a singleton, say $\bigcap_k B_k = \{\xx\}$, whose unique member $\xx$ lies in $K$. The point $\xx$ is called the \emph{outcome} of the game. A set $S \subseteq K$ is said to be \emph{$(\alpha,\beta)$-winning on $K$} if Alice has a strategy guaranteeing that the outcome lies in $S$, regardless of the way Bob chooses to play. It is said to be \emph{$\alpha$-winning on $K$} if it is $(\alpha,\beta)$-winning on $K$ for every $0 < \beta < 1$, and \emph{winning on $K$} if it is $\alpha$-winning on $K$ for some $0 < \alpha < 1$.
\end{definition}

\begin{definition}
A set $K \subset \R^d$ is said to be \emph{hyperplane diffuse} if there exists $\beta > 0$ such that for all $0 < \rho \leq 1$ and $\xx\in K$ and for any (affine) hyperplane $\LL \subset \R^d$, we have
\[
B(\xx,\rho)\cap K \butnot \thickvar{\LL}{\beta\rho} \neq\emptyset,
\]
where $\thickvar{\LL}{\epsilon}$ denotes the closed $\epsilon$-thickening of $\LL$, i.e. $\thickvar{\LL}{\epsilon} = \{\yy\in\R^d : \dist(\yy,\LL) \leq \epsilon\}$.
\end{definition}

\begin{proposition}[{\cite[Theorem 2.5 + Proposition 4.7]{BFKRW}}]
\label{propositionBFKRW}
Let $K \subset \R^d$ be a closed and hyperplane diffuse set. Then $\BA_d\cap K$ is winning on $K$.
\end{proposition}

This result remains true if $\BA_d$ is replaced by any \emph{hyperplane absolute winning} set, see \cite[p.4]{BFKRW} for the definition. The same applies to all of the results of this paper.

To state in a clearer way the lower bound for the Hausdorff dimension of a winning set discovered by the second-named author \cite[Theorem 3.1]{Fishman}, we recall the definition of the lower Assouad dimension of a set. (The lower Assouad dimension has been given several names in the literature, see \cite[p.6688]{Fraser}, but the name ``lower Assouad dimension'' seems by far the most natural to us.)

\begin{definition}
\label{definitionlowerA}
Given $\rho > 0$ and $S \subset \R^d$, we let $N_\rho(S)$ denote the cardinality of any maximal $\rho$-separated subset of $S$. (Choosing a different maximal $\rho$-separated set will not change $N_\rho(S)$ by more than a constant factor.) The \emph{lower Assouad dimension} of a nonempty closed set $K \subset \R^d$, denoted $\underline\AD(K)$, is the supremum of $s \geq 0$ such that there exists a constant $c > 0$ such that for all $\xx\in K$ and $0 < \beta,\rho \leq 1$, we have
\[
N_{\beta\rho}\big(B(\xx,\rho)\cap K\big) \geq c \beta^{-s}.
\]
Equivalently,
\begin{align*}
\underline\AD(K)
&= \liminf_{\beta \to 0} \inf_{0 < \rho \leq 1} \inf_{\xx\in K} \frac{\log N_{\beta\rho}(B(\xx,\rho)\cap K)}{-\log(\beta)}\\
&= \liminf_{\beta \to 0} \liminf_{\rho \to 0} \inf_{\xx\in K} \frac{\log N_{\beta\rho}(B(\xx,\rho)\cap K)}{-\log(\beta)}
\end{align*}
(cf. Appendix \ref{sectionassouad}).
\end{definition}

The lower Assouad dimension is the smallest of the standard fractal dimensions. In particular, if $K$ is closed then $\underline{\AD}(K) \leq \HD(K)$, see e.g. \cite[Lemma 2.2]{KLV}. Note that unlike most notions of dimension, the lower Assouad dimension is not monotone: a subset may have larger lower Assouad dimension than the set it is contained in.

The essential idea of the following result is found in \cite[Theorem 3.1]{Fishman}. We include the proof for completeness.

\begin{proposition}
\label{propositionfishman}
Let $K\subset \R^d$ be closed and let $S \subset K$ be winning on $K$. Then
\[
\HD(S) \geq \underline\AD(K),
\]
where $\underline\AD$ denotes the lower Assouad dimension.
\end{proposition}
\begin{proof}
Let $\delta = \underline\AD(K)$, and fix $\epsilon > 0$. Then by definition, there exists a constant $c = c_\epsilon > 0$ such that for all $\xx\in K$, $0 < \beta \leq 1/4$, and $0 < \rho \leq 1$, we have
\[
N_{3\beta\rho}\big(B(\xx,(1 - \beta)\rho)\cap K\big) \geq N = N(\beta) \df \lfloor c_\epsilon \beta^{-(\delta - \epsilon)}\rfloor.
\]
Now let $\alpha > 0$ be chosen so that $S$ is $\alpha$-winning, and fix $0 < \beta \leq 1/2$. For each ball $A = B(\xx,\rho)$, we choose a $3\beta\rho$-separated sequence $\yy^{(1)}(A),\ldots,\yy^{(N)}(A) \in B(\xx,(1 - \beta)\rho) \cap K$, and we let $f_i(A)$ denote the ball centered at $\yy^{(i)}(A)$ of radius $\beta\rho$. Then the balls $f_1(A),\ldots,f_N(A)$ are contained in $A$ and separated by distances of at least $\beta\rho$. Moreover, each ball $f_i(A)$ is a legal move for Bob to make in response to Alice playing $A$ as her move.

Now fix a winning strategy for Alice to win the $(\alpha,\beta)$-game, and we will consider the family of counterstrategies for Bob such that whenever Alice plays a ball $A_k$, Bob responds by playing one of the balls $f_1(A_k),\ldots,f_N(A_k)$. We fix Bob's initial ball (chosen to have radius less than 1), and for each function $\omega:\N\to E \df \{1,\ldots,N(\beta)\}$, we consider the counterstrategy in which Bob responds to Alice's $k$th move $A_k$ by choosing the ball $B_{k + 1} = f_{\omega(k)}(A_k)$. We denote the outcome of this counterstrategy by $\pi(\omega)$, so that $\pi:E^\N \to K$. The separation conditions on the balls $f_1(A),\ldots,f_N(A)$ guarantee that
\[
\dist(\pi(\omega),\pi(\tau)) \asymp (\alpha\beta)^{|\omega\wedge\tau|} \all \omega,\tau\in E^\N,
\]
where $|\omega\wedge\tau|$ denotes the length of the longest common initial segment of $\omega$ and $\tau$. Thus the uniform Bernoulli measure on $\pi(E^\N)$ is Ahlfors $s(\beta)$-regular, where
\[
s(\beta) \df \frac{\log N(\beta)}{-\log(\alpha\beta)} = \frac{-(\delta - \epsilon) \log(\beta) + O(1)}{-\log(\alpha\beta)} \tendsto{\beta\to 0} \delta - \epsilon.
\]
It follows that $\HD(\pi(E^\N)) \geq s(\beta)$. Now, since each element of $\pi(E^\N)$ is the outcome of a game where Alice played her winning strategy, we have $\pi(E^\N) \subset S$ and thus $\HD(S) \geq s(\beta) \to \delta - \epsilon$. Since $\epsilon$ was arbitrary, we have $\HD(S) \geq \delta$.
\end{proof}

Combining Propositions \ref{propositionBFKRW} and \ref{propositionfishman} gives:

\begin{corollary}
\label{corollaryBA}
Let $K \subset \R^d$ be a closed and hyperplane diffuse set. Then
\[
\HD(\BA_d\cap K) \geq \underline\AD(K).
\]
\end{corollary}

With this result in mind, we can see how the following theorem is a strengthening of Theorem \ref{maintheorem2}:

\begin{theorem}
\label{maintheorem3}
Let $\Lambda_\Phi \subset [0,1]^d$ be a self-affine sponge, and let $\pp$ be an irreducible good probability measure on $E$ with distinct Lyapunov exponents. Then there exists a sequence of strongly Lalley--Gatzouras sponges $\Lambda_{\Psi_N} \subset \Lambda_\Phi$ that are hyperplane diffuse and satisfy
\[
\underline\AD(\Lambda_{\Psi_N}) \to \HD(\nu_\pp).
\]
\end{theorem}


Here, following \cite{DasSimmons1}, we use the term ``Lalley--Gatzouras'' to refer to a certain class of sponges that includes the carpets considered by \name{Steven}{Lalley} and \name{Dimitrios}{Gatzouras}:

\begin{definition}[Cf. {\cite[Definition 3.6]{DasSimmons1}}]
\label{definitionLG}
A self-affine sponge $\Lambda_\Phi$ is \emph{Lalley--Gatzouras} (resp. \emph{strongly Lalley--Gatzouras}) if there exists a permutation $\sigma$ of $D$ such that both of the following hold:
\begin{itemize}
\item (Coordinate ordering condition) For all $\aa\in E$, we have
\[
|\phi_{\aa,\sigma(1)}'| > |\phi_{\aa,\sigma(2)}'| > \cdots > |\phi_{\aa,\sigma(d)}'|;
\]
\item (Disjointness condition) The coordinate sets $\sigma(I_{\leq i})$ ($i = 1,\ldots,d$) are all good (resp. strongly good), where $I_{\leq i} \df \{1,\ldots,i\}$.
\end{itemize}
\end{definition}

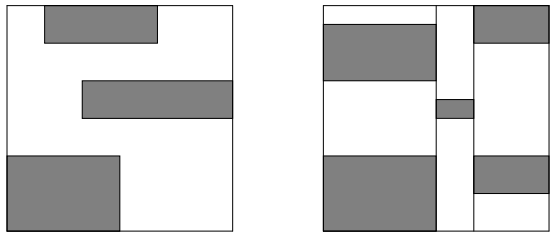
\begin{figure}
\scalebox{0.5}{
\begin{tikzpicture}[line cap=round,line join=rounr]
\clip(-1,-1) rectangle (7,7);
\fill[gray] (0,0) -- (0,2) -- (3,2) -- (3,0) -- cycle;
\fill[gray] (1,5) -- (1,6) -- (4,6) -- (4,5) -- cycle;
\fill[gray] (2,3) -- (2,4) -- (6,4) -- (6,3) -- cycle;
\draw (0,0)-- (0,6);
\draw (6,0)-- (6,6);
\draw (0,0)-- (6,0);
\draw (0,6)-- (6,6);
\draw (0,0) -- (0,2) -- (3,2) -- (3,0) -- cycle;
\draw (1,5) -- (1,6) -- (4,6) -- (4,5) -- cycle;
\draw (2,3) -- (2,4) -- (6,4) -- (6,3) -- cycle;
\end{tikzpicture}}
\scalebox{0.5}{
\begin{tikzpicture}[line cap=round,line join=rounr]
\clip(-1,-1) rectangle (7,7);
\fill[gray] (0,0) -- (0,2) -- (3,2) -- (3,0) -- cycle;
\fill[gray] (0,4) -- (0,5.5) -- (3,5.5) -- (3,4) -- cycle;
\fill[gray] (4,1) -- (4,2) -- (6,2) -- (6,1) -- cycle;
\fill[gray] (4,5) -- (4,6) -- (6,6) -- (6,5) -- cycle;
\fill[gray] (3,3) -- (3,3.5) -- (4,3.5) -- (4,3) -- cycle;
\draw (0,0)-- (0,6);
\draw (6,0)-- (6,6);
\draw (0,0)-- (6,0);
\draw (0,6)-- (6,6);
\draw (0,0) -- (0,2) -- (3,2) -- (3,0) -- cycle;
\draw (0,4) -- (0,5.5) -- (3,5.5) -- (3,4) -- cycle;
\draw (4,1) -- (4,2) -- (6,2) -- (6,1) -- cycle;
\draw (4,5) -- (4,6) -- (6,6) -- (6,5) -- cycle;
\draw (3,3) -- (3,3.5) -- (4,3.5) -- (4,3) -- cycle;
\draw (3,0)-- (3,6);
\draw (4,0)-- (4,6);
\end{tikzpicture}}
\caption{Generating templates for two carpets satisfying the coordinate ordering condition. The picture on the right also satisfies the disjointness condition, making it a Lalley--Gatzouras carpet.}
\label{figureBMLGB}
\end{figure}

\section{Some results on Lalley--Gatzouras sponges}
\label{sectionLG}

In this section, we give a necessary and sufficient condition for the hyperplane diffuseness of a strongly Lalley--Gatzouras sponge $\Lambda_\Phi$ (which will be fixed throughout the section), as well as a formula for the the lower Assouad dimension of $\Lambda_\Phi$.  For conceptual completeness we also state the formula for the upper Assouad dimension of $\Lambda_\Phi$, which is defined as follows:

\begin{definition}
The \emph{upper Assouad dimension} of a nonempty closed set $K \subset \R^d$, denoted $\overline\AD(K)$, is the infimum of $s \geq 0$ such that there exists $C > 0$ such that for all $\xx\in K$ and $0 < \beta,\rho \leq 1$, we have
\[
N_{\beta\rho}\big(B(\xx,\rho)\cap K\big) \leq C \beta^{-s},
\]
where $N_{\beta\rho}$ is as in Definition \ref{definitionlowerA}. Equivalently,
\begin{align*}
\overline\AD(K)
&= \limsup_{\beta \to 0} \sup_{0 < \rho \leq 1} \sup_{\xx\in K} \frac{\log N_{\beta\rho}(B(\xx,\rho)\cap K)}{-\log(\beta)}\\
&= \limsup_{\beta \to 0} \limsup_{\rho \to 0} \sup_{\xx\in K} \frac{\log N_{\beta\rho}(B(\xx,\rho)\cap K)}{-\log(\beta)}
\end{align*}
(cf. Appendix \ref{sectionassouad}).
\end{definition}

\begin{definition}
\label{definitionunifirred}
The sponge $\Lambda_\Phi$ is called \emph{uniformly irreducible} if for all $i\in D$ and $\aa\in E$, there exists $\bb\in E$ such that $b_i \neq a_i$ but \irredabi, where the partial order $\prec$ is the same as in Definition \ref{definitionirreducible}.
\end{definition}

For example, the irreducible Bara\'nski sponge $\Lambda_{\Phi^{(1)}}$ appearing in Example \ref{examplecarpet} is not uniformly irreducible, since if $i = 1$ and $\aa = (1,1)$ then there is no $\bb\in E^{(1)}$ such that $b_i\neq a_i$ but \irredabi. On the other hand, if $E^{(3)} = \{(0,0),(1,1),(2,0),(2,1)\}$, then the corresponding Bara\'nski sponge $\Lambda_{\Phi^{(3)}}$ is uniformly irreducible.

In the remainder of this section, we assume without loss of generality that the permutation $\sigma$ appearing in Definition \ref{definitionLG} is trivial, i.e. that the orders $\prec$ and $\leq$ on $D$ are equivalent. (Note that this is not true for the sponges of Example \ref{examplecarpet}.)

\begin{proposition}
\label{propositiondiffuse}
The sponge $\Lambda_\Phi$ is hyperplane diffuse if and only if it is uniformly irreducible.
\end{proposition}
We will use the backwards direction of this proposition in the proof of Theorem \ref{maintheorem3}; we include the proof of the forwards direction for completeness.
\begin{proof}[Proof of backwards direction]
Fix $\omega\in E^\N$ and $0 < \rho \leq 1$, and let $\xx = \pi(\omega)$; we will prove that
\begin{equation}
\label{Hdiffuse}
\Lambda_\Phi \cap B(\xx,\rho) \butnot\thickvar\LL{\beta\rho} \neq \emptyset
\end{equation}
for any hyperplane $\LL$, where $\beta > 0$ is an appropriate constant. For each $i\in D$, let $N_i\in\N$ be the smallest number such that
\begin{equation}
\label{Nidef}
\prod_{n = 1}^{N_i} |\phi_{\omega_n,i}'| \leq \rho,
\end{equation}
and note that by the coordinate ordering condition, we have $N_1 \geq N_2 \geq \cdots \geq N_d$. Let
\begin{equation}
\label{notations}
\begin{split}
[\aa]_i &= \{\bb\in E : a_i = b_i\}\\
[\omega\given N]_i &\df \{\tau\in E^\N : \tau_n \in [\omega_n]_i \all n \leq N\}\\
B_\omega(N_1,\ldots,N_d) &\df \bigcap_{i\in D} [\omega\given N_i]_i.
\end{split}
\end{equation}
Then by \eqref{Nidef},
\[
B_\omega(N_1,\ldots,N_d) \subset \pi^{-1}\big(B(\pi(\omega),\rho)\big).
\]
(Here for convenience we work with the max norm on $\R^d$.) Now fix $i\in D$ and let $\aa = \omega_{N_i + 1}$. Let $\bb\in E$ be as in Definition \ref{definitionunifirred}. Define the point $\tau\in E^\N$ as follows: let $\tau_n = \omega_n$ for all $n\neq N_i + 1$, and let $\tau_{N_i + 1} = \bb$. Finally, let $\yy^{(i)} = \pi(\tau)$.

Fix $j\in D$. If $a_j = b_j$, then clearly $\tau\in [\omega\given N_j]_j$. On the other hand, if $a_j \neq b_j$, then $j\geq i$ and thus $N_j \leq N_i$, which implies $\tau\in [\omega\given N_j]_j$. So either way we have $\tau\in [\omega\given N_j]_j$, and thus $\tau\in B_\omega(N_1,\ldots,N_d)$. Consequently, $\yy^{(i)} \in B(\xx,\rho)$.

If $j < i$, then $a_j = b_j$ and thus $x_j = y^{(i)}_j$. On the other hand, since $a_i \neq b_i$ and since $I_{\leq i}$ is strongly good, we have $\max_{j\leq i} |y^{(i)}_j - x_j| \asymp \rho$. Combining these facts gives $|y^{(i)}_i - x_i| \asymp \rho$. Since $i,j$ were arbitrary, this means that the matrix $M = (y^{(i)}_j - x_j)_{i,j} = \sum_{i\in D} (\ee^{(i)})\cdot(\yy^{(i)} - \xx)^T$ is upper triangular and its diagonal entries are asymptotic to $\rho$, while all of its entries are bounded in magnitude by $\rho$. Here $\ee^{(i)}$ denotes the (column) vector whose $i$th entry is $1$ (and whose other entries are $0$). This implies that
\[
\|M^{-1}\| \lesssim \rho^{-1}.
\]
So if $\vv\in\R^d$ is any unit vector, then $\|M\vv\| \gtrsim \rho$ and thus there exists $i\in D$ such that $|(\yy^{(i)} - \xx) \cdot \vv| \gtrsim \rho$. It follows that there exists a constant $\beta > 0$ (independent of $\xx,\rho,\vv$) such that
\begin{equation}
\label{2betarho}
\dist(\yy^{(i)} - \xx,\vv^\perp) > 2\beta\rho.
\end{equation}
Now if $\LL \subset \R^d$ is a hyperplane, then we can write $\LL = \pp + \vv^\perp$ for some $\pp\in\R^d$ and some unit vector $\vv$, and then \eqref{2betarho} implies that
\[
\max(\dist(\xx,\LL),\dist(\yy^{(i)},\LL)) > \beta\rho.
\]
Since $\xx,\yy^{(i)} \in \Lambda_\Phi \cap B(\xx,\rho)$, this demonstrates \eqref{Hdiffuse}, completing the proof.
\end{proof}

\begin{proof}[Proof of forwards direction]
By contradiction suppose that $\Lambda_\Phi$ is not uniformly irreducible. Then there exist $i \in D$ and $\aa\in E$ such that for all $\bb\in E$ satisfying $b_i \neq a_i$, we have $\{j\in D : a_j \neq b_j\} \nsubset \{j \in D : j \geq i\}$. Let $\omega = \aa^\infty$, $\xx = \pi(\omega)$, and $\LL = \xx + \sum_{j\neq i} \R \ee^{(j)}$. We claim that for all $0 < \rho \leq 1$,
\begin{equation}
\label{nondiffuse}
\Lambda_\Phi \cap B(\xx,\rho) \subset \NN(\LL,C \rho^\alpha),
\end{equation}
where $C > 0$ and $\alpha > 1$ are constants. This implies that $\Lambda_\Phi$ is not hyperplane diffuse.

Indeed, fix $0 < \rho \leq 1$, and for each $j\in D$, let $N_j\in\N$ be the largest number such that
\[
\prod_{n = 1}^{N_j} |\phi_{\omega_n,j}'| = |\phi_{\aa,j}'|^{N_j} \geq \epsilon^{-1} \rho,
\]
where
\begin{equation}
\label{epsilondef}
\epsilon = \min_{\substack{I = I_{\leq j} \\ j\in D}} \min_{\substack{\aa,\bb\in \pi_I(E) \\ \text{distinct}}} \dist\big(\phi_{I,\aa}([0,1]^I),\phi_{I,\bb}([0,1]^I)\big).
\end{equation}
Since $\Lambda_\Phi$ satisfies the disjointness condition, we have $\epsilon > 0$. As before we have $N_1 \geq N_2 \geq \cdots \geq N_d$. We let the notations $[\aa]_i$, $[\omega\given N]_i$, and $B_\omega(N_1,\ldots,N_d)$ be as in the previous proof, but this time our definition of $N_j$ implies that
\[
\pi^{-1}\big(B(\pi(\omega),\rho)\big) \subset B_\omega(N_1,\ldots,N_d).
\]
Fix $\tau \in \pi^{-1}\big(B(\pi(\omega),\rho)\big)$, and we will estimate $\dist(\yy,\LL)$, where $\yy = \pi(\tau)$. Fix $n \leq N_{i - 1}$ (with the convention that $N_0 = \infty$), and let $\bb = \tau_n$. Since $\tau \in B_\omega(N_1,\ldots,N_d)$, we have $b_j = \tau_{n,j} = \omega_{n,j} = a_j$ for all $j < i$. By the definition of $i$, this implies that $b_i = a_i$, and thus $\tau_{n,i} = \omega_{n,i}$ for all $n \leq N_{i - 1}$. It follows that
\[
\dist(\yy,\LL) = |y_i - x_i| \leq \prod_{n = 1}^{N_{i - 1}} |\phi_{\omega_n,i}'| = |\phi_{\aa,i}'|^{N_{i - 1}}.
\]
If $i = 1$, then we have shown that $\dist(\yy,\LL) = 0$. Suppose $i > 1$. Since $\Lambda_\Phi$ satisfies the coordinate ordering condition, we have
\[
\alpha \df \frac{\log|\phi_{\aa,i}'|}{\log|\phi_{\aa,i - 1}'|} > 1.
\]
On the other hand, the definition of $N_{i - 1}$ implies that $|\phi_{\aa,i - 1}'|^{N_{i - 1} + 1} < \epsilon^{-1} \rho$, so
\[
\dist(\yy,\LL) \leq |\phi_{\aa,i - 1}'|^{N_{i - 1} \alpha} \lesssim \rho^\alpha,
\]
demonstrating \eqref{nondiffuse}.
\end{proof}

If $\Lambda_\Phi$ is irreducible but not uniformly irreducible, then in the next section we will show that $\Lambda_\Phi$ contains uniformly irreducible subsponges. So in this case, even though $\Lambda_\Phi$ is not hyperplane diffuse it contains hyperplane diffuse subsets. On the other hand:

\begin{proposition}
\label{propositionnoHDsubsets}
If $\Lambda_\Phi$ is reducible, then it contains no hyperplane diffuse subsets.
\end{proposition}
In this case, the techniques of Section \ref{sectiongames} cannot possibly be used to prove that $\BA_d \cap \Lambda_\Phi$ is large, because these techniques rely on finding hyperplane diffuse subsets of $\Lambda_\Phi$ (with sufficiently large lower Assouad dimension).
\begin{proof}
Since $\Lambda_\Phi$ is reducible, there exists $i\in D$ such that for all $\aa,\bb\in E$ such that $a_i \neq b_i$, there exists $j < i$ such that $a_j \neq b_j$. Fix $\xx = \pi(\omega) \in \Lambda_\Phi$, and let $\LL = \xx + \sum_{j\neq i} \R \ee^{(j)}$. Repeating the second paragraph of the proof of the forwards direction of Proposition \ref{propositiondiffuse} shows that \eqref{nondiffuse} holds for all $0 < \rho \leq 1$. (One minor change is needed: after fixing $n \leq N_{i - 1}$, we let $\aa = \omega_n$. Since the condition on $i$ now holds for all $\aa,\bb\in E$, the subsequent argument is still valid.) Since $\xx\in \Lambda_\Phi$ was arbitrary, \eqref{nondiffuse} implies that no subset of $\Lambda_\Phi$ is hyperplane diffuse.
%
\end{proof}

Next, we compute the upper and lower Assouad dimensions of $\Lambda_\Phi$. For each $i\in D$, let $\pi_i = \pi_{I_{<i}} = \pi_{I_{\leq i - 1}}$, and for each $\aa\in \pi_i(E)$, consider the ``fiber IFS''
\begin{equation}
\label{Phiia}
\Phi_{i,\aa} = (\phi_{i,b})_{b\in E_{i,\aa}} \text{ where } E_{i,\aa} = \{b\in A_i :(\aa,b)\in \pi_{i + 1}(E)\}.
\end{equation}
Note that $\pi_1(E) = \{\emptyset\}$ and $\pi_{d + 1}(E) = E$. We let $\dim(\Phi_{i,\aa})$ denote the dimension of the limit set of the IFS $\Phi_{i,\aa}$. (Since the limit set is Ahlfors regular, it does not matter what notion of fractal dimension we use.)

\begin{theorem}
\label{theoremassdimLG}
Recall that $\Lambda_\Phi$ denotes a strongly Lalley--Gatzouras sponge such that the orders $\prec$ and $\leq$ on $D$ are equivalent. For each $i\in D$ let
\begin{align*}
\underline\delta_i &= \min_{\aa\in \pi_i(E)} \dim(\Phi_{i,\aa})\\
\overline\delta_i &= \max_{\aa\in \pi_i(E)} \dim(\Phi_{i,\aa}).
\end{align*}
Then
\begin{align} \label{lowerAformula}
\underline\AD(\Lambda_\Phi)
&= \sum_{i\in D} \underline\delta_i\\ \label{upperAformula}
\overline\AD(\Lambda_\Phi)
&= \sum_{i\in D} \overline\delta_i
\end{align}
\end{theorem}
The case $d = 2$ of Theorem \ref{theoremassdimLG} (i.e. carpets) was proven recently by \name{Jonathan}{Fraser} \cite[Theorems 2.12 and 2.13]{Fraser}, and the case of Sierpi\'nski sponges by \name{Jonathan}{Fraser} and \name{Douglas}{Howroyd} \cite{FraserHowroyd}. The case of general Bara\'nski sponges appears to be more subtle.

We will prove only \eqref{lowerAformula}, since that is the equation that we will need in the proof of Theorem \ref{maintheorem3}. (To be precise, we only need the $\geq$ direction of \eqref{lowerAformula}.) The proof of \eqref{upperAformula} is similar.

\begin{proof}[Proof of $\geq$ direction]
Fix $\omega\in E^\N$ and $0 < \beta,\rho \leq 1$, and let $\xx = \pi(\omega)$. For each $i \in D$, let $N_i\in\N$ be the smallest number such that \eqref{Nidef} holds. Fix $i \in D$, and consider the space
\[
X_i \df \prod_{n = 1}^{N_i} \{\omega_{n,i}\} \times \prod_{n = N_i + 1}^{N_{i - 1}} E_{i,\pi_i(\omega_n)}.
\]
Here we use the convention that $N_0 = \infty$. Let $r_i$ be the map
\[
r_i([\tau\given N]) = \prod_{n = 1}^N |\phi_{\tau_n,i}'|,
\]
i.e. up to a constant, $r_i$ sends a cylinder in $X_i$ to the diameter of the $i$th coordinate of its image under the coding map. Here $\tau$ denotes any element of $E^\N$.

\begin{claim}
\label{claimdelta}
If $[\tau\given N_{i - 1}]$ is a maximal-length cylinder of $X_i$, then
\begin{equation}
\label{maximalcylinders}
r_i([\tau\given N_{i - 1}]) \lesssim \rho^{1 + \delta}
\end{equation}
where $\delta > 0$ is a constant.
\end{claim}
\begin{proof}
We have
\begin{align*}
r_i([\tau\given N_{i - 1}])
&= \left(\prod_{n = 1}^{N_i} |\phi_{\omega_n,i}'|\right)
\left(\prod_{n = N_i + 1}^{N_{i - 1}} |\phi_{\tau_n,i}'|\right) 
\leq \lambda_+^{N_{i - 1} - N_i} \rho,
\end{align*}
where $\lambda_+ = \max_{i\in D} \max_{b\in A_i} |\phi_{i,b}'| < 1$. Also,
\begin{align*}
r_i([\omega\given N_{i - 1}])
&= \left(\prod_{n = 1}^{N_i} |\phi_{\omega_n,i}'|\right)
\left(\prod_{n = N_i + 1}^{N_{i - 1}} |\phi_{\omega_n,i}'|\right)
\geq \lambda_-^{N_{i - 1} - N_i} \rho,
\end{align*}
where $\lambda_- = \min_{i\in D} \min_{b\in A_i} |\phi_{i,b}'| > 0$. On the other hand, if
\[
\alpha \df \max_{i\in D} \max_{\aa\in E} \frac{\log|\phi_{\aa,i}'|}{\log|\phi_{\aa,i - 1}'|} > 1,
\]
then
\begin{align*}
r_i([\omega\given N_{i - 1}])
&= \prod_{n = 1}^{N_{i - 1}} |\phi_{\omega_n,i}'|
\leq \prod_{n = 1}^{N_{i - 1}} |\phi_{\omega_n,i - 1}'|^\alpha \asymp \rho^\alpha.
\end{align*}
Combining these inequalities gives
\[
\lambda_-^{N_{i - 1} - N_i} \lesssim \rho^{\alpha - 1}
\]
and thus
\[
r_i([\tau\given N_{i - 1}]) \lesssim \rho^{1 + (\alpha - 1)\log(\lambda_+)/\log(\lambda_-)}.
\QEDmod\qedhere\]
\end{proof}

In the remainder of the proof we assume that $\rho \leq (C^{-1} \beta)^{1/\delta}$, where $C$ is the implied constant of \eqref{maximalcylinders}. Then
\[
r_i([\tau\given N_{i - 1}]) \leq \beta \rho
\]
for all maximal-length cylinders $[\tau\given N_{i - 1}]$ in $X_i$. It follows that the collection
\[
\PP_i = \{[\tau\given N] : r_i([\tau\given N]) \leq \beta\rho < r_i([\tau\given N - 1])\}
\]
is a partition of $X_i$.

Now let $s_i = \underline\delta_i$. By the definition of $\underline\delta_i$, we have
\[
\sum_{b\in E_{i,\aa}} |\phi_{i,b}'|^{s_i} \geq 1 \all \aa\in \pi_i(E),
\]
and thus the map $r_i^{s_i}$ is subadditive on cylinders of length at least $N_i$. So
\[
\rho^{s_i} \asymp r_i^{s_i}([\omega\given N_i]) \leq \sum_{P\in \PP_i} r_i^{s_i}(P) \asymp (\beta\rho)^{s_i} \#(\PP_i),
\]
i.e. $\#(\PP_i) \gtrsim \beta^{-{s_i}}$. Now let $\PP = \prod_{i = 1}^d \PP_i$, and let $\P(E^\N)$ denote the power set of $E^\N$. Define the map $\iota:\PP\to \P(E^\N)$ as follows:
\begin{equation}
\label{Pdef}
\iota([\tau_1\given M_1],\ldots,[\tau_d\given M_d])
= \bigcap_{i\in D} [\tau_i\given M_i]_i \subset B_\omega(N_1,\ldots,N_d).
\end{equation}
Then the sets $\pi(\iota(\bfP))$ ($\bfP\in \PP$) are contained in $B(\xx,\rho)$ and separated by distances $\gtrsim \beta\rho$. Thus
\[
N_{\beta\rho}(B(\xx,\rho)\cap \Lambda_\Phi) \gtrsim \#(\PP) = \prod_{i = 1}^d \#(\PP_i) \gtrsim \prod_{i = 1}^d \beta^{-s_i} = \beta\wedge\left(-\sum_{i = 1}^d \underline\delta_i\right),
\]
assuming that $\rho \leq (C^{-1} \beta)^{1/\delta}$. Here $\beta\wedge s$ denotes $\beta$ raised to the power of $s$. Taking the infimum over $\xx\in K$, the liminf as $\rho\to 0$, and then the liminf as $\beta\to 0$ completes the proof.
\end{proof}

\begin{proof}[Proof of $\leq$ direction]
For each $i\in D$, let $\aa^{(i)} \in E$ be chosen so that
\begin{equation}
\label{aidef}
\dim(\Phi_{i,\pi_i(\aa^{(i)})}) = \underline\delta_i.
\end{equation}
Fix $0 < \rho \leq 1$, and define the sequence $N_1,\ldots,N_d$ by backwards recursion: if $N_{i + 1},\ldots,N_d$ are defined, then let $N_i$ be the smallest integer such that
\[
\prod_{j = i}^d |\phi_{\aa^{(j)},i}'|^{N_j - N_{j + 1}} \leq \rho.
\]
Then let $\omega\in E^\N$ be the infinite word defined by the formula
\begin{equation}
\label{omegandef}
\omega_n = \aa^{(i)} \all i\in D \all n = N_i + 1,\ldots,N_{i - 1}.
\end{equation}
Note that for each $i\in D$, $N_i$ is the smallest integer that satisfies \eqref{Nidef}, i.e. $N_i$ has the same value in this proof as it did in the proof of the $\geq$ direction. Fix $i\in D$, and let $X_i$ and $r_i$ be as in the proof of the $\geq$ direction. By \eqref{omegandef} and \eqref{aidef}, $r_i$ is additive on cylinders rather than merely being subadditive. Moreover, since the sets $\pi(\iota(\bfP))$ ($\bfP\in \PP$) defined by \eqref{Pdef} have diameter $\lesssim \beta\rho$ and form a cover of $B(\xx,\lambda_- \epsilon\rho)\cap \Lambda_\Phi$ (here $\epsilon$ is as in \eqref{epsilondef}), we have $\asymp$ in the last calculation rather than just $\gtrsim$:
\begin{equation}
\label{Nbetarho}
N_{\beta\rho}\big(B(\xx,\lambda_- \epsilon\rho)\cap \Lambda_\Phi\big) \asymp \beta\wedge\left(-\sum_{i = 1}^d \underline\delta_i\right).
\end{equation}
The quantifiers on this statement are: for all $\rho$, there exists $\xx$ such that \eqref{Nbetarho} holds for all $\beta \geq C \rho^\delta$, where $C$ is the implied constant of \eqref{maximalcylinders} and $\delta$ is as in Claim \ref{claimdelta}. In particular, by varying  $\rho$ we can make $\beta$ arbitrarily small while still retaining \eqref{Nbetarho}. This completes the proof.
\end{proof}

\section{Proof of the main theorems}
\label{sectionproof}

In this section we prove Theorem \ref{maintheorem3}, thus indirectly proving Theorems \ref{maintheorem} and \ref{maintheorem2}, which are consequences of Theorem \ref{maintheorem3}.

We recall that in Theorem \ref{maintheorem3}, $\pp$ is an irreducible good measure with distinct Lyapunov exponents. Without loss of generality, we suppose that the orders $\prec_\pp$ and $\leq$ on $D$ are equivalent, i.e. that $\chi_1(\pp) < \chi_2(\pp) < \ldots < \chi_d(\pp)$.

By perturbing the measure $\pp$, we may assume that $\pp(\aa) > 0$ for all $\aa\in E$. Since $\Lambda_\Phi$ is irreducible with respect to $\pp$, this implies that
\begin{equation}
\label{irredh}
\text{h}_\pp(I_{\leq i}\given I_{<i}) \df \int \log\frac{\pp([\aa]_{I_{<i}})}{\pp([\aa]_{I_{\leq i}})}\;\dee\pp(\aa) > 0
\end{equation}
for all $i\in D$. Here $[\aa]_I = \bigcap_{j\in I} [\aa]_j = \{\bb\in E : b_j = a_j \all j\in I\}$.

Now fix $\epsilon > 0$ and $N\in\N$, and let $S = S_N \subset E^N$ be the set of all words $\omega\in E^N$ satisfying
\begin{equation}
\label{lyapiterate}
(1 - \epsilon) N \chi_i(\pp) \leq -\log|\phi_{\omega,i}'| = \sum_{j = 1}^N -\log|\phi_{\omega_j,i}| \leq (1 + \epsilon) N \chi_i(\pp)\\
\end{equation}
and
\begin{equation}
\label{entropyiterate}
\log\frac{\mu([\omega]_{I_{<i}})}{\mu([\omega]_{I_{\leq i}})}
= \sum_{j = 1}^N \log\frac{\pp([\omega_j]_{I_{<i}})}{\pp([\omega_j]_{I_{\leq i}})}
\geq (1 - \epsilon) N \text{h}_\pp(I_{\leq i}\given I_{<i}),
\end{equation}
where $\mu = \mu_N = \pp^N$. Here the notation is slightly different from in \eqref{notations}:
\begin{align*}
[\omega]_I &= \{\tau\in E^N : \tau_n \in [\omega_n]_I \all n\leq N\}.
\end{align*}
By the law of large numbers, we have $\lim_{N\to\infty} \mu_N(S_N) = 1$, so if $N$ is sufficiently large then $\mu(S) \geq 1 - \epsilon$.

Now define a sequence of sets $T_d \supset T_{d - 1} \supset \cdots \supset T_0$ as follows: $T_d = S$, and if $T_i$ is defined then let
\[
T_{i - 1} = \{\omega\in T_i : \mu(T_i\cap [\omega]_{I_{<i}}) \geq \epsilon \mu([\omega]_{I_{<i}})\}.
\]
Letting
\[
\II_i = \{[\aa]_{I_{<i}} : \aa\in E\} = \{\pi_i^{-1}(\aa) : \aa\in \pi_i(E)\},
\]
we have
\begin{align*}
\mu(T_i\butnot T_{i - 1})
&= \sum_{\substack{P\in \II_i \\ \mu(T_i\cap P) < \epsilon \mu(P)}} \mu(T_i\cap P)
\leq \sum_{\substack{P\in \II_i \\ \mu(T_i\cap P) < \epsilon \mu(P)}} \epsilon\mu(P) \leq \epsilon\mu(E^N) = \epsilon,
\end{align*}
so $\mu(T_0) \geq 1 - (d + 1)\epsilon$. In particular, if $\epsilon$ is small enough then $\mu(T_0) > 0$, and in particular $T_0 \neq \emptyset$.

\begin{claim}
\label{claimT0}
For all $\omega\in T_0$, we have $T_0\cap [\omega]_{I_{<i}} = T_i\cap [\omega]_{I_{<i}}$.
\end{claim}
\begin{subproof}
Fix $\tau\in T_i\cap [\omega]_{I_{<i}}$; we will prove by backwards induction that $\tau\in T_j$ for all $j = i,i - 1,\ldots,0$. Fix $j\leq i$, and suppose that $\tau\in T_j$. Since $\tau \in [\omega]_{I_{<i}} \subset [\omega]_{I_{<j}}$, we have $[\tau]_{I_{<j}} = [\omega]_{I_{<j}}$. Since $\omega \in T_0 \subset T_{j - 1}$, this shows that $\mu(T_i\cap [\tau]_{I_{<i}}) \geq \epsilon \mu([\tau]_{I_{<i}})$, and thus $\tau\in T_{j - 1}$.
\end{subproof}

Note that in this claim, it was crucial that we defined the sequence $(T_i)_0^d$ by backward recursion rather than by forward recursion, since we needed the fact that $[\omega]_{I_{<i}} \subset [\omega]_{I_{<j}}$ for all $j \leq i$. Combining Claim \ref{claimT0} with the definition of $(T_i)_0^d$ yields
\begin{equation}
\label{T0}
\mu(T_0\cap [\omega]_{I_{<i}}) \geq \epsilon \mu([\omega]_{I_{<i}}) \all i\in D \all \omega\in T_0.
\end{equation}
Now let $\tau\in E^* = \bigcup_{n=0}^\infty E^n$ be a word of fixed length (independent of $N$) such that
\begin{equation}
\label{taudef}
\phi_\tau([0,1]^d) \subset (0,1)^d,
\end{equation}
and consider the diagonal IFS $\Psi = \Psi_N = (\psi_\omega)_{\omega\in T_0}$, where for each $\omega\in T_0$, we write
\[
\psi_\omega = \phi_\omega\circ \phi_\tau.
\]
Clearly, $\Lambda_\Psi \subset \Lambda_\Phi$. To bound $\underline{\AD}(\Lambda_\Psi)$ from below using Theorem \ref{theoremassdimLG}, we fix $i\in D$ and $\aa\in \pi_i(T_0)$. We have
\begin{align*}
\dim(\Psi_{i,\aa})
&\geq \frac{\log\#(E_{i,\aa})}{\displaystyle\max_{b\in E_{i,\aa}} (-\log|\psi_{i,b}'|)}
\end{align*}
where the ``fiber IFS'' $\Psi_{i,\aa} = (\psi_b)_{b\in E_{i,\aa}}$ is as in \eqref{Phiia}. Now by \eqref{lyapiterate},
\[
\max_{b\in E_{i,\aa}} (-\log|\psi_{i,b}'|) \leq (1 + \epsilon) N \chi_i(\pp) + (-\log|\phi_{\tau,i}'|),
\]
and on the other hand
\begin{align*}
\#(E_{i,\aa}) &\geq \frac{\mu(T_0\cap \pi_i^{-1}(\aa))}{\max_{b\in E_{i,\aa}} \mu(T_0\cap \pi_{i + 1}^{-1}(\aa,b))}\\
&\geq \frac{\epsilon\mu(\pi_i^{-1}(\aa))}{\max_{b\in E_{i,\aa}} \mu(\pi_{i + 1}^{-1}(\aa,b))} \by{\eqref{T0}}\\
&\geq \epsilon\exp\big((1 - \epsilon) N \text{h}_\pp(I_{\leq i}\given I_{<i})\big). \by{\eqref{entropyiterate}}
\end{align*}
So
\begin{equation}
\label{dimfiber}
\dim(\Psi_{i,\aa})
\geq \delta_i(N,\epsilon) \df \frac{(1 - \epsilon) N \text{h}_\pp(I_{\leq i}\given I_{<i}) + \log(\epsilon)}{(1 + \epsilon) N \chi_i(\pp) - \log|\phi_{\tau,i}'|}
\end{equation}
and thus by Theorem \ref{theoremassdimLG},
\[
\underline{\AD}(\Psi_N) \geq \sum_{i\in D} \delta_i(N,\epsilon) \tendsto{N\to\infty} \frac{1 - \epsilon}{1 + \epsilon} \sum_{i\in D} \frac{\text{h}_\pp(I_{\leq i}\given I_{<i})}{\chi_i(\pp)} \tendsto{\epsilon \to 0} \sum_{i\in D} \frac{\text{h}_\pp(I_{\leq i}\given I_{<i})}{\chi_i(\pp)}\cdot
\]
The right-hand side is equal to $\HD(\nu_\pp)$ by the Ledrappier--Young formula (e.g. \cite[Proposition 2.16]{DasSimmons1}). So to complete the proof, we need to show that the sponge $\Lambda_{\Psi_N}$ is strongly Lalley--Gatzouras and hyperplane diffuse for all $N$ sufficiently large. The coordinate ordering condition follows from \eqref{lyapiterate} and the fact that $\pp$ has distinct Lyapunov exponents. Since $\pp$ is good, so is $\Lambda_{\Psi_N}$, and \eqref{taudef} implies that $\Lambda_{\Psi_N}$ is in fact strongly good. Finally, combining \eqref{irredh} with the calculation preceding \eqref{dimfiber} shows that
\[
\min_{i\in D} \min_{\aa\in \pi_i(E)} \#(E_{i,\aa}) \geq 2
\]
for all $N$ sufficiently large, and it is easy to check that this is equivalent to $\Psi_N$ being uniformly irreducible. So by Proposition \ref{propositiondiffuse}, $\Lambda_{\Psi_N}$ is hyperplane diffuse for all $N$ sufficiently large.

\section{Open questions}

We conclude with a list of open questions.

\begin{question}
Let $\Lambda_\Phi \subset [0,1]^d$ be an irreducible Bara\'nski sponge with distinguishable coordinates, such that $\DynD(\Lambda_\Phi) < \HD(\Lambda_\Phi)$ (cf. \cite{DasSimmons1}). Does $\Lambda_\Phi \cap \BA_d$ necessarily have full Hausdorff dimension in $\Lambda_\Phi$ (as opposed to just full dynamical dimension as guaranteed by Theorem \ref{maintheorem})? Alternatively, can it be shown that it is impossible to prove this using the techniques of this paper, by showing that $\Lambda_\Phi$ does not necessarily have hyperplane diffuse subsets of sufficiently large lower Assouad dimension?
\end{question}

\begin{question}
Can Schmidt's game be used to show that $\BA_d$ has full dimension in some fractal defined by a dynamical system which is both non-conformal and nonlinear?
\end{question}

\begin{question}
Is there any fractal $\Lambda$ defined by a smooth dynamical system (e.g. the limit set of a $C^1$ IFS) such that $\Lambda\cap \BA_d$ is nonempty but does not have full dimension in $\Lambda$?
\end{question}

\appendix
\section{Equivalent formulas for the upper and lower Assouad dimensions}
\label{sectionassouad}

In this paper we have used the formulas
\begin{align*}
\underline\AD(K)
= \underline\delta_1
&\df \liminf_{\beta \to 0} \inf_{0 < \rho \leq 1} \inf_{\xx\in K} \frac{\log N_{\beta\rho}(B(\xx,\rho)\cap K)}{-\log(\beta)}\\
= \underline\delta_2
&\df \liminf_{\beta \to 0} \liminf_{\rho \to 0} \inf_{\xx\in K} \frac{\log N_{\beta\rho}(B(\xx,\rho)\cap K)}{-\log(\beta)}\\
\overline\AD(K)
= \overline\delta_1
&\df \limsup_{\beta \to 0} \sup_{0 < \rho \leq 1} \sup_{\xx\in K} \frac{\log N_{\beta\rho}(B(\xx,\rho)\cap K)}{-\log(\beta)}\\
= \overline\delta_2
&= \limsup_{\beta \to 0} \limsup_{\rho \to 0} \sup_{\xx\in K} \frac{\log N_{\beta\rho}(B(\xx,\rho)\cap K)}{-\log(\beta)}
\end{align*}
for the lower and upper Assouad dimension, respectively. It is clear from the definitions that $\underline\AD(K) = \underline\delta_1$ and $\overline\AD(K) = \overline\delta_1$, but it is less clear that $\underline\delta_1 = \underline\delta_2$ and $\overline\delta_1 = \overline\delta_2$, so we prove this now. For brevity we only prove the equality $\underline\delta_1 = \underline\delta_2$, as the proof of the equality $\overline\delta_1 = \overline\delta_2$ is similar.

Obviously $\underline\delta_1\leq\underline\delta_2$, so we fix $s < \underline\delta_2$, and we will show that $\underline\delta_1 \geq s$. Choose $0 < \beta \leq 1$ small enough so that
\[
\liminf_{\rho\to 0} \inf_{\xx\in K} \frac{\log N_{\beta\rho}(B(\xx,\rho)\cap K)}{-\log(\beta/4)} > s,
\]
and then choose $\rho_0 > 0$ small enough so that
\[
\inf_{\xx\in K} \frac{\log N_{\beta\rho}(B(\xx,\rho)\cap K)}{-\log(\beta/4)} \geq s \all 0 < \rho \leq \rho_0.
\]
Then every ball $B(\xx,\rho)$ centered at a point in $K$ of radius $\leq \rho_0$ contains a $\beta\rho$-separated set of cardinality at least $(\beta/4)^{-s}$. Now the balls of radius $\beta\rho/2$ centered at the points of this set are disjoint and contained in $B(\xx,2\rho)$. Letting $\kappa = 2\rho$, we see that every ball $B(\xx,\kappa)$ with $\xx\in K$ and $\kappa \leq 2\rho_0$ contains at least $(\beta/4)^{-s}$ disjoint balls of radius $(\beta/4) \kappa$. Iterating, every such ball contains at least $(\beta/4)^{-ns}$ disjoint balls of radius $(\beta/4)^n \kappa$. Now fix $\xx\in K$ and $0 < r \leq R \leq 1$. Let $r' = \min(\rho_0,r)$ and $R' = \min(\rho_0,R)$, and let $n\geq 0$ be chosen so that $(\beta/4)^n \geq r'/R'$ but $(\beta/4)^n \asymp r'/R' \asymp r/R$. Then by the above argument, $B(\xx,R') \subset B(\xx,R)$ contains at least $(\beta/4)^{-ns} \asymp (r/R)^{-s}$ disjoint balls of radius $(\beta/4)^n R' \geq r'$. If $r' = r$, this shows that $N_r(B(\xx,R)) \gtrsim (R/r)^s$, and if $r' < r$, then $r > \rho_0$ and thus $N_r(B(\xx,R)) \geq 1 \asymp (R/r)^s$. By the definition of the lower Assouad dimension, this implies that $\underline\delta_1 = \underline\AD(K) \geq s$.

\bibliographystyle{amsplain}

\bibliography{bibliography}

\end{document}